\documentclass{amsart}

\usepackage{amsmath}%
\usepackage{amsfonts}%
\usepackage{amssymb}%
\newtheorem{theorem}{Theorem}

\newtheorem{lemma}[theorem]{Lemma}
\newtheorem*{theoremA}{Theorem A}

\newcommand{\Q}{\mathbb{Q}}
\newcommand{\R}{\mathbb{R}}

\newcommand{\Z}{\mathbb{Z}}

\numberwithin{equation}{section}

\begin{document}

\author{C.~Gutierrez}
\address{Instituto de Ci\^encias Matem\'aticas e de Computa\c c\~ao, Universidade de S\~ao Paulo, S\~ao Carlos - SP, Brazil}
\email{gutp@icmc.usp.br}
\thanks{The first author was partially supported by FAPESP Grant 03/03107-9 and by CNPq Grants 470957/2006-9 and
306328/2006-2.}

\author{S.~Lloyd}
\address{School of Mathematics and Statistics, University of New South Wales, Sydney NSW, Australia}
\email{s.lloyd@unsw.edu.au}

\author{B.~Pires}
\address{Departamento de F\'isica e Matem\'atica, Faculdade de Filosofia, Ci\^encias e Letras da 
Universidade de S\~ao Paulo, Ribeir\~ao Preto - SP, Brazil}
\email{benito@ffclrp.usp.br}

\title[Affine IETs with flips and wandering intervals]
{Affine interval exchange transformations\\ 
with flips and wandering intervals}

\date{February 29 2008}

\begin{abstract}
There exist uniquely ergodic
affine interval exchange transformations of [0,1] with flips
having wandering intervals and such that the support of
the invariant measure is a Cantor set.
\end{abstract}

\maketitle

\section{Introduction}

Let $N$ be a compact subinterval of either $\R$ or the circle $S^1$, and let $f:N\to N$ be piecewise
continuous. We say that a subinterval $J\subset N$ is a {\it wandering interval} of the map $f$ if the forward iterates $f^n(J)$, $n=0,1,2,\ldots$ are pairwise disjoint intervals, each not reduced to a point, and 
the $\omega$-limit set of $J$ is an infinite set.

A great deal of information about the topological dynamics of a map $f:N\to N$ is revealed when one knows 
whether $f$ has wandering intervals. This turns out to be a subtle question whose answer depends on
both the topological and regularity properties of the map $f$.

The question of the existence of wandering intervals first arose when $f$ is a diffeomorphism of the circle
$S^1$. The Denjoy counterexample shows that even a
$C^1$ diffeomorphism $f:S^1\to S^1$ may have wandering intervals. This behaviour
is ruled out when $f$ is smoother. More specifically, if $f$ is a $C^1$ diffeomorphism of the circle such that the
logarithm of its derivative has bounded variation then $f$ has no wandering intervals \cite{Den}. In this case the topological dynamics of $f$ is simple: if $f$ has no periodic points, then $f$ is topologically conjugate to a rotation.

The first results ensuring the absence of wandering intervals on continous maps satisfying some smoothness conditions
were provided by Guckenheimer \cite{Guc}, Yoccoz \cite{Yoc}, and Blokh and Lyubich \cite{B-L}. Later on, de Melo et al.~\cite{MMS} generalised these results proving that if $N$ is compact and $f:N\to N$ is a $C^2$-map with non-flat critical points then $f$ has no wandering intervals.
Concerning discontinous maps, Berry and Mestel \cite{B-M} found a condition which excludes wandering intervals in Lorenz
maps --- interval maps with a single discontinuity. Of course, conservative maps and, in particular, interval exchange transformations, admit no wandering intervals. We consider the following generalisation of interval exchange transformations.

Let $0\le a<b$ and let $\{a,b\}\subset D\subset [a,b]$ be a discrete set containing $n$
points. We say that an injective, continuously differentiable map $T:[a,b]\to [a,b]$ defined
on $\mathcal{D}\,(T)=[a,b]\setminus D$ is an {\it affine interval exchange transformation of $n$-subintervals}, shortly an {\it $n$-AIET}, if $\vert DT\vert$ is a positive,
locally constant function such that $T([a,b])$ is all of $[a,b]$ except for finitely many points. We also assume
that the points in $D\setminus\{a,b\}$ are non-removable discontinuities of $T$. We say that an AIET
is {\it oriented} if $DT>0$, otherwise we say that $T$ has {\it flips}. An {\it isometric IET
of $n$ subintervals}, shortly an $n$-IET, is an $n$-AIET satisfying $\vert DT\vert=1$ everywhere.

Levitt \cite{Lev} found an example of a non-uniquely ergodic oriented AIET with wandering intervals. Therefore there are Denjoy counterexamples of arbitrary smoothness. Gutierrez and Camelier \cite{C-G} constructed an AIET with wandering intervals that is semiconjugate to a self-similar IET. The regularity of conjugacies between AIETs and self-similar IETs is examined by Cobo \cite{Cob} and by Liousse and Marzougui \cite{L-M}. Recently, Bressaud, Hubert and Maass \cite{BHM} provided sufficient conditions for a self-similar IET to have an AIET with a wandering interval semiconjugate to it. 

In this paper we present an example of a self-similar IET with flips having the particular
property that we can apply the main result of the work \cite{BHM} to obtain a $5$-AIET with flips
semiconjugate to the referred IET and having densely distributed wandering intervals. The AIET so obtained
is uniquely ergodic \cite{Ve1} (see \cite{Mas,Ve2}) and the support of the invariant measure is a Cantor set.

A few remarks are due in order to place this example in context. The existence of minimal non-uniquely ergodic AIETs with flips and wandering intervals would follow by the same argument of Levitt \cite{Lev},
provided we knew a minimal non-uniquely ergodic IET with flips. However, no example of minimal non-uniquely
ergodic IET with flips is known, although it is possible to insert flips in the example of Keane \cite{Kea} (for oriented IETs) to get a transitive non-uniquely IET with flips having saddle-connections. Computational evaluations indicate that it is impossible to obtain, via Rauzy induction, examples of self-similar $4$-IETs with flips meeting the hypotheses of \cite{BHM}, despite
this being possible in the case of oriented $4$-IETs (see \cite{C-G,Cob}). Thus the example we present here is the simplest possible, in the sense that wandering intervals do not occur for AIETs with flips semiconjugate to a self--similar IET, obtained via Rauzy induction, defined on a smaller number of intervals.

\section{Self-similar interval exchange transformations}

Let $T:[a,b]\to [a,b]$ be an $n $-AIET defined on $[a,b]\setminus D$,
where $D=\{x_0,\ldots,x_n\}$ and $a=x_0<x_1<\ldots<x_{n-1}<x_n=b$. 
Let $\beta_i\neq 0$ be the derivative of $T$ on $(x_{i-1},x_i)$, $i=1,2\ldots,n$. 
We shall refer to
$$ x=(x_0,x_1,\ldots,x_n) $$
as the {\it D-vector} of $T$ (i.e. the domain-of-definition-vector of
$T$). The vectors
\begin{eqnarray*}
\gamma = (\log |\beta_1|, \log |\beta_2|, \ldots, \log
|\beta_n|) \quad \mbox{and} \quad \tau = \left(\frac{\beta_1}{|\beta_1|},
\frac{\beta_2}{|\beta_2|},\ldots,\frac{\beta_n}{|\beta_n|}\right)
\end{eqnarray*}
will be called the {\it log-slope-vector} and the {\it flips-vector
of $T$}, respectively.  
Notice that $T$ has flips if and only if some coordinate of $\tau$  is equal to $-1$.
Let
\begin{eqnarray*}
\{z_1,\ldots,z_n\}=\left\{T\left(\frac{x_0+x_1}{2}\right),T\left(\frac{x_1+x_2}{2}\right),\ldots,T\left(\frac{x_{n-1}+x_n}{2}\right)\right\}
\end{eqnarray*}
be such that $ 0<z_1<z_2<\ldots<z_n<1$; we define the {\it permutation $\pi$
associated to $T$} as the one that takes $i\in\{1,2,\ldots,n\}$
to  $\pi(i)=j$  if and only if $z_{j}=T((x_{i-1}+x_i)/2)$.

It should be remarked that an AIET $E:[a,b]\to [a,b]$ with flips-vector $\tau\in\{-1,1\}^n$ and  which has the zero vector as the
log-slope-vector is an IET (with
flips-vector $\tau$) and conversely. Let $J=[c,d]$ be a proper subinterval of $[a,b]$.
We say that the IET $E$ is {\it self-similar} (on $J$) if
there exists an orientation preserving affine map $L:\R\to\R$ such
that $L(J)=[a,b]$  and $L\circ\widetilde{E}=E\circ L$, where $\widetilde{E}:J\to J$ denotes the IET induced by $E$ and $L(\mathcal{D}(\widetilde{E}))\subset\mathcal{D}(E)$. A self-similar IET $E:[a,b]\to [a,b]$ on a proper subinterval
$J\subset [a,b]$ will be denoted by $(E,J)$.

Given an AIET $E:[a,b]\to [a,b]$, the {\it orbit} of $p\in [a,b]$ is the set
$$O(p)=\{E^n(p)\mid n\in\Z \:,{\rm and}\: p\in\mathcal{D}(E)\}.$$
The AIET $E$ is called {\it transitive} if there exists an orbit of $E$ that is dense in $[a,b]$. We say that the
orbit of $p\in [a,b]$ is {\it finite} if $\#(O(p))<\infty$. In this way, a point $p\in [a,b]-
(\mathcal{D}(E)\cup\mathcal{D}(E^{-1}))$ is said to have a finite orbit. A transitive AIET is {\it minimal} if it has no finite orbits.

Let $E:[a,b]\to [a,b]$ be an IET  with  D-vector
$(x_0,x_1,\cdots,x_n) $. Denote by $J=[c,d]$ a proper
subinterval of $[a,b]$. Suppose that $E$ is self-similar (on $J$);
so there exists IET $\widetilde{E}:J\to J$  such that $L(J)=[a,b]$ and $L\circ\widetilde{E}=E\circ L$.
Given $i=0,1,\cdots,n$, let $y_i=L^{-1}(x_i)$. In this way,
the sequence of discontinuities of $\widetilde{E}$ is  $\{y_1,\cdots,y_{n-1}\}$.

We say that a non-negative matrix is {\it quasi-positive} if some power of it is a positive matrix. 
A non-negative matrix is quasi-positive if and only if it is both irreducible and aperiodic.
Let $A$ be an $n\times n$ non-negative matrix whose entries
are: 
\begin{eqnarray*}
A_{ji} = \#\{ 0 \leq k \leq N_i : E^k((y_{i-1}, y_i))
\subset (x_{j-1}, x_j) \},
\end{eqnarray*}
where $N_i$ is the least 
non-negative integer such that for some $y\in (y_{i-1},y_i)$ (and
therefore for all $y\in (y_{i-1},y_i)$), $E^{N_i+1}(y)\in J$. We
shall refer to $A$ as the {\it matrix associated to} $(E,J)$. 
Being self-similar, $E$ is also transitive, which implies the quasi-positivity of $A$.
Hence, by the Perron-Frobenius Theorem \cite{Gan}, $A$ possesses exactly one probability
right eigenvector $\alpha\in\Lambda_n$, where
$$\Lambda_{n}=\{\lambda=(\lambda_1,\ldots,\lambda_n)\mid\lambda_i>0,\,\forall i\}.$$
Moreover, the eigenvalue $\mu$ corresponding to $\alpha$
is simple, real and  greater than $1$ and, also,
all other eigenvalues of $A$ have absolute value less than
$\mu$. It was proved by Veech \cite{Ve1} (see also \cite{Mas,Ve2}) that every self-similar IET 
is minimal and uniquely ergodic. Furthermore,
following Rauzy \cite{Rau}, we conclude that
\begin{eqnarray*}
\alpha =
(x_1 - x_0, x_2 - x_1, \cdots, x_n - x_{n-1}).
\end{eqnarray*}

\section{The theorem of Bressaud, Hubert and Maass}

Let $A\in SL_n(\Z)$ and let $\Q[t]$ be the ring of polynomials with rational coefficients in one variable. We say that two
real eigenvalues $\theta_1$ and $\theta_2$ of $A$ are {\it conjugate} if there exists an irreducible polynomial
$f\in\Q[t]$ such that $f(\theta_1)=f(\theta_2)=0$. We say that an AIET  $T$ of $[0,1]$ is {\it semiconjugate}
(resp. {\it conjugate}) to an IET $E$ of $[0,1]$ if there exists a  non-decreasing (resp. bijective) continous map $h:[0,1]\to [0,1]$
such that $h(\mathcal{D}(T))\subset\mathcal{D}(E)$ and $E\circ h=h\circ T$. 

\begin{theorem}[Bressaud, Hubert and Maass, 2007]\label{BHMthm} Let $J$ be a proper subinterval of $[0,1]$, \mbox{$E:[0,1]\to [0,1]$} be an interval exchange transformation self-similar on $J$ and let $A$ be the matrix associated to $(E,J)$.
Let $\theta_1$ be the Perron-Frobenius eigenvalue of $A$. Assume that
$A$ has a real eigenvalue $\theta_2$ such that
\begin{itemize}
 \item [(1)] $1<\theta_2\: (<\theta_1)$;
 \item [(2)] $\theta_1$ and  $\theta_2$ are conjugate.
 \end{itemize}
Then there exists an affine interval exchange transformation $T$ of $[0,1]$ with wandering intervals that is semiconjugate
to $E$.
\end{theorem}
\proof This theorem was proved in \cite{BHM} for oriented IETs. The same proof holds word for word for IETs
with flips. In this case, the AIET $T$ inherits its flips from the IET $E$ through the semiconjugacy
previously constructed therein.
\endproof

\section{The interval exchange transformation $E$}

In this section we shall present the IET we shall use to construct the AIET with flips and wandering intervals.
We shall need the Rauzy induction \cite{Rau} to obtain a minimal, self-induced IET whose associated matrix satisfies
all the hypotheses of Theorem \ref{BHMthm}.

Let $ \alpha = (\alpha_1, \alpha_2, \alpha_3, \alpha_4, \alpha_5)\in\Lambda_5$ 
be the probability (i.e. each $ \alpha_i >0$ and $ |\alpha| = \alpha_1 +  \alpha_2 + \alpha_3 + \alpha_4 +
\alpha_5 = 1 $ ) Perron-Frobenius right eigenvector of the matrix
\begin{eqnarray*}
A = \left(\begin{array} {ccccc}
         2 & 4 & 6 & 5 & 2 \\
         0 & 2 & 1 & 1 & 1 \\
         0 & 0 & 3 & 2 & 0 \\
         1 & 2 & 2 & 2 & 1 \\
         1 & 3 & 5 & 4 & 2
 \end{array}\right).
\end{eqnarray*}
The eigenvalues $\theta_1,\theta_2,\rho_1,\rho_2,\rho_3$ of $A$  are
real and have approximate values:
\begin{eqnarray*}
\theta_1=7.829,\:\theta_2=1.588,\:\rho_1=1,\:\rho_2=0.358,\:\rho_3=0.225
\end{eqnarray*}
and $ \alpha=(\alpha_1,\alpha_2,\alpha_3,\alpha_4,\alpha_5)$, the probability right
eigenvector associated to $\theta_1$, has approximate value
\begin{eqnarray*}
\alpha = (0.380, 0.091, 0.070, 0.170, 0.289).
\end{eqnarray*}
Notice that $\alpha_1+\alpha_2+\alpha_3+\alpha_4+\alpha_5=1$. In what follows we represent
a permutation $\pi$ of the set $\{1,2,\ldots,n\}$ by the $n$-tuple $\pi=(\pi(1),\pi(2),\ldots,
\pi(n))$.

We consider the iet $E:[0, 1]\to [0, 1]$  which is
determined by the following conditions:

\begin{enumerate}
  \item [(1)] $E$ has the D-vector $ x = (x_0, x_1, x_2, x_3, x_4, x_5)$,
  where
  $$
  x_0=0;\quad x_i=\sum_{k=1}^i \alpha_k,\quad i=1,\ldots,5;
  $$
  \item [(2)] $E$ has associated permutation $(5,3,2,1,4)$;
  \item [(3)] $E$ has flips-vector $(-1,-1,1,1,-1)$.
\end{enumerate}

\begin{lemma}\label{example}
The map $E$  is self-similar on the interval $J=[0, 1/\theta_1]$,  and $A$
is precisely the matrix associated to $(E,J)$.
\end{lemma}
\begin{proof}
We apply the Rauzy algorithm (see [Rau]) to the IET $E$. We represent $E:I\to I$ by the pair $E^{(0)}=(\alpha^{(0)},p^{(0)})$ where $\alpha^{(0)}=\alpha$ is its length vector and $p^{(0)}=(-5,-3,2,1,-4)$ is its signed permutation, obtained by elementwise multiplication of its permutation $(5,3,2,1,4)$ and flips-vector $(-1,-1,1,1,-1)$. We shall apply the Rauzy procedure fourteen times, obtaining IETs $E^{(k)}=(\alpha^{(k)},p^{(k)})$, $k=0,\ldots,14$, with D-vector $x^{(k)}$ given by $x^{(k)}_0=0$; and $x^{(k)}_i=\sum_{j=1}^i \alpha^{(k)}_j$,
for $i=1,2,\ldots,5$.

\begin{table}[htbp]
	\centering
		\begin{tabular}{|c| *{4}{r @{\:}}r|c|}
\hline
 $k$ && \multicolumn{3}{c}{$p^{(k)}$} && $t^{(k)}$ \\
\hline
 0&-5&-3& 2& 1&-4& 1\\
 1& 4&-5&-3& 2& 1& 0\\
 2& 5&-2&-4& 3& 1& 1\\
 3& 5& 1&-2&-4& 3& 1\\
 4& 5& 3& 1&-2&-4& 1\\
 5& 5&-4& 3& 1&-2& 0\\
 6&-2&-5& 4& 1&-3& 1\\
 7&-2& 3&-5& 4& 1& 0\\
 8&-3& 4&-2& 5& 1& 1\\
 9&-3& 4&-2& 5& 1& 1\\
10&-3& 4&-2& 5& 1& 0\\
11&-4& 5&-3& 2& 1& 1\\
12&-4& 5& 1&-3& 2& 1\\
13&-4& 5& 2& 1&-3& 0\\
14&-5&-3& 2& 1&-4& 1\\
\hline
		\end{tabular}
	\caption{Rauzy cycle with associated matrix $A$.}
	\label{tab:Table}
\end{table}

Given an IET $E^{(k)}$, defined on an interval $[0,L^{(k)}]$ and represented by the pair $(\alpha^{(k)},p^{(k)})$, the IET $E^{(k+1)}$ is defined to be the map induced on the interval $[0,L^{(k+1)}]$ by $E^{(k)}$, where $L^{(k+1)}=L^{(k)}-\min\,\{\alpha^{(k)}_5,\alpha^{(k)}_{s}\}$ and $s$ is such that
$\vert p^{(k)}_{\,n}(s)\vert=5$. We say that the type $t^{(k)}$ of $E^{(k)}$ is $0$ if $\alpha^{(k)}_5>\alpha^{(k)}_s$ and $1$ if $\alpha^{(k)}_5<\alpha^{(k)}_{s}$. Notice that $\sum_{i=1}^5 \alpha^{(k)}_i=L^{(k)}$.

The new signed permutations $p^{(k)}$, obtained by this procedure are given in Table~\ref{tab:Table}, along with the type $t^{(k)}$ of $E^{(k)}$. The length vector $\alpha^{(k+1)}$ is obtained from $\alpha^{(k)}$ by the equation $\alpha^{(k)}=M(p^{(k)},t^{(k)}).\alpha^{(k+1)}$, where $M(p^{(k)},t^{(k)})\in SL_n(\Z)$ is a certain elementary matrix (see \cite{GLMPZ}). Moreover, we have that
$$
M(p^{(0)},t^{(0)}).\cdots.M(p^{(13)},t^{(13)})=A.
$$
Thus $\alpha^{(14)}=A^{-1}.\alpha^{(0)}=\alpha^{(0)}/\theta_1$, and $J=[0,L^{(14)}]$. Notice that $p^{(14)}=p^{(0)}$, and so we have a Rauzy cycle: $R^{(14)}$ and $R^{(0)}$ have the same flips-vector and permutation. Hence $\widetilde{E}=E^{(14)}$ is a $1/\theta_1$-scaled copy of $E=E^{(0)}$, and so $E$ is self-similar on the interval $J$.

As remarked before, since $E$ self-similar, we have that the matrix associated to $(E,J)$
is quasi-positive. In fact, we have that $A$ is the matrix associated to $(E,J)$. To see that, 
for $i\in\{0,\ldots,5\}$, let $y_i=x_i/\theta_1$ be the points of discontinuity for $\widetilde{E}$. Table~\ref{tab:itin} shows the itinerary $I(i)=\{I(i)_k\}_{k=1}^{N_i}$ of each interval $(y_{i-1},y_i)$, where \mbox{$N_i=\min\,\{n>1:E^{n+1}((y_{i-1},y_i))\subset J\}$} and $I(i)_k=r$ if and only if $E^k((y_{i-1},y_i))\subset (x_{r-1},x_r)$.

\begin{table}[htbp]
	\centering
\begin{tabular}{|c|c| *{17}{c@{\:}}|}
\hline
 $i$ & $N_i$ &&\multicolumn{14}{c}{$I(i)$} && \\
\hline
 1& 4& 1& 5& 1& 4&  &  &  &  &  &  &  &  &  &  &  &  &   \\
 2&11& 1& 5& 2& 1& 4& 1& 5& 2& 1& 5& 4&  &  &  &  &  &   \\
 3&17& 1& 5& 2& 1& 4& 1& 5& 3& 1& 5& 3& 1& 5& 3& 1& 5& 4 \\
 4&14& 1& 5& 2& 1& 4& 1& 5& 3& 1& 5& 3& 1& 5& 4&  &  &   \\
 5& 6& 1& 5& 2& 1& 5& 4&  &  &  &  &  &  &  &  &  &  &   \\
\hline
\end{tabular}
	\caption{Itineraries $I(i)$, $i\in\{1,\ldots,5\}$.}
	\label{tab:itin}
\end{table}

The number of times that $j$ occurs in $I(i)$, for $i,j\in\{1,\ldots,5\}$, is precisely $A_{ji}$ and thus $A$ is the matrix associated to the pair $(E,J)$ as required.
\end{proof}

\begin{theoremA} There exists a uniquely ergodic
affine interval exchange transformation of $[0,1]$ with flips
having wandering intervals and such that the support of
the invariant measure is a Cantor set.
\end{theoremA}
\begin{proof} By construction, the matrix $A$ associated to $(E,J)$ satisfies hypothesis $(1)$
of \mbox{Theorem \ref{BHMthm}}. The characteristic polynomial $p(t)$ of $A$ can be written as the
product of two irreducible polynomials over $\Q[t]$: 
$$p(t)=(1-t) (1-8t+18t^2-10t^3+t^4).$$
Thus the eigenvalues $\theta_1$ and $\theta_2$ are zeros of the same irreducible polynomial
of degree four and so are conjugate. Hence, $A$ also verifies hypothesis $(2)$ of \mbox{Theorem \ref{BHMthm}},
which finishes the proof.
\end{proof}

Note that for an AIET $T$, the forward and backward iterates of a wandering interval $J$ form a pairwise disjoint collection of intervals. Moreover, when $T$ is semiconjugate to a transitive IET, as is the case in Theorem A, the $\alpha$-limit set and $\omega$-limit set of $J$ coincide. 

\bibliographystyle{amsplain}

\end{document}